\documentclass[12pt]{article}
\usepackage{amssymb}
\usepackage{amsthm}
\usepackage{amsmath}
\usepackage{graphicx}
\usepackage{enumerate}

\newtheorem{theorem}{Theorem}
\newtheorem{definition}[theorem]{Definition}
\newtheorem{lemma}[theorem]{Lemma}
\newtheorem{proposition}[theorem]{Proposition}

\newtheorem{example}[theorem]{Example}
\newtheorem{corollary}[theorem]{Corollary}
\title{Orthogonality and complementation in the lattice of subspaces of a finite-dimensional vector space over a finite field}
\author{Ivan~Chajda and Helmut~L\"anger}
\date{}
\begin{document}
\footnotetext[1]{Support of the research by the Austrian Science Fund (FWF), project I~4579-N, and the Czech Science Foundation (GA\v CR), project 20-09869L, as well as by \"OAD, project CZ~02/2019, is gratefully acknowledged.}
\maketitle
\begin{abstract}
We investigate the lattice $\mathbf L(\mathbf V)$ of subspaces of an $m$-dimensional vector space $\mathbf V$ over a finite field ${\rm GF}(q)$ with a prime power $q=p^n$. It is well-known that this lattice is modular and that orthogonality is an antitone involution. The lattice $\mathbf L(\mathbf V)$ satisfies the Chain condition and we determine the number of covers of its elements, especially the number of its atoms. We characterize when orthogonality is a complementation and hence when $\mathbf L(\mathbf V)$ is orthomodular. For $m>1$ and $p\nmid m$ we show that $\mathbf L(\mathbf V)$ contains a $(2^m+2)$-element (non-Boolean) orthomodular lattice as a subposet. Finally, for $q$ being a prime and $m=2$ we characterize orthomodularity of $\mathbf L(\mathbf V)$ by a simple condition.
\end{abstract}



The lattice of subspaces of a given vector space was studied by several authors from various points of view. In particular, for (possibly infinite-dimensional) vector spaces over the field of complex numbers such lattices serve as an algebraic axiomatization of the logic of quantum mechanics. It was shown that such lattices are orthomodular and, if the vector space has finite dimension, even modular. The question arises if something similar holds for vector spaces over finite fields. An attempt in this direction was done by Eckmann and Zabey (\cite{EZ}). An interesting structure in such a lattice is the sublattice of closed subspaces. Unfortunately, this lattice need not be a sublattice of the lattice of all subspaces and, moreover, these lattice even need not be orthomodular. These facts imply that these lattices in general cannot be used in the logic of quantum mechanics (see \cite{EZ}). However, subspaces of finite-dimensional vector spaces turn out to be closed. It turns out that in such a case the lattice of subspaces sometimes has a nice structure and hence may be used in the axiomatization of logics similarly as lattices of topologically closed subspaces of Hilbert spaces over the complex numbers are used in the logic of quantum mechanics.

Throughout the paper we consider finite-dimensional vector spaces $\mathbf V$ over a finite field ${\rm GF}(q)$. Assume $\dim\mathbf V=m>1$ and $q=p^n$ for some prime $p$.

The paper is organized in the following way. First we derive some conditions which are satisfied by the lattice $\mathbf L(\mathbf V)$ of subspaces of $\mathbf V$ and we obtain a certain relationship between $m$ and $q$. Then we characterize those $\mathbf V$ for which $\mathbf L(\mathbf V)$ is orthomodular. The lattice $\mathbf L(\mathbf V)$ turns out to be orthomodular if and only if orthogonality is a complementation. We show that $\mathbf L(\mathbf V)$ contains Boolean subalgebras of size $2^m$ and that in case $p\nmid m$, $\mathbf L(\mathbf V)$ contains a $(2^m+2)$-element (non-Boolean) orthomodular lattice as a subposet. We show that in case $m=2$ the lattice $\mathbf L(\mathbf V)$ is fully determined by $q$. Unfortunately, we do not know any conditions characterizing those finite modular lattices which are isomorphic to $\mathbf L(\mathbf V)$ for some finite-dimensional vector space $\mathbf V$ over a finite field provided the dimension of $\mathbf V$ is greater than $2$. Hence we like to encourage the readers of this paper to continue in this research.

In the whole paper let $\mathbb N$ denote the set of all positive integers. Let $V$ denote the universe of $\mathbf V$. Without loss of generality assume $V=({\rm GF}(q))^m$. For $\vec a=(a_1,\ldots,a_m),\vec b=(b_1,\ldots,b_m)\in V$ and $M\subseteq V$ put
\begin{align*}
        \vec a\vec b & :=\sum_{i=1}^ma_ib_i, \\
   \vec a\perp\vec b & :\Leftrightarrow\vec a\vec b=0, \\
             M^\perp & :=\{\vec x\in V\mid\vec x\perp\vec y\text{ for all }\vec y\in M\}, \\
    \langle M\rangle & :=\text{ linear subspace of }\mathbf V\text{ generated by }M, \\
        L(\mathbf V) & :=\text{ set of all linear subspaces of }\mathbf V, \\
\mathbf L(\mathbf V) & :=(L(\mathbf V),+,\cap,{}^\perp,\{\vec 0\},V).
\end{align*}
As usual, by a {\em basis} of $\mathbf V$ we mean a linearly independent generating set of $\mathbf V$. It is well-known that any $m$ linearly independent vectors of $V$ form a basis of $\mathbf V$. Moreover, it is well-known that $\mathbf L(\mathbf V)$ is a modular lattice with an antitone involution, see e.g.\ Theorems~15 and 16 in \cite{CL19}. Moreover, this lattice is {\em paraorthomodular} (see \cite{GLP} for this concept and for several corresponding results) because it satisfies the condition
\[
U\subseteq W\text{ and }U^\perp\cap W=\{\vec0\}\text{ imply }U=W,
\]
see \cite{CL19} for the proof. It is well known that every bounded modular lattice with an antitone involution which is, moreover, a complementation is already orthomodular. Recall from \cite{Be} that an {\em orthomodular lattice} is a bounded lattice $(L,\vee,\wedge,{}',0,1)$ with an antitone involution which is a complementation such that
\[
x\leq y\text{ implies }y=x\vee(x'\wedge y)
\]
($x,y\in L$). The above arguments show that $\mathbf L(\mathbf V)$ is orthomodular if and only if $^\perp$ is a complementation. Hence we want to investigate when $^\perp$ is a complementation. For this purpose we define: The vector $\vec a\in V$ is called {\em isotropic} if $\vec a\neq\vec0$ and $\vec a\vec a=0$. For $n>1$ let $\mathbf M_n$ denote the modular lattice with the following Hasse diagram (see Fig.~1):

\vspace*{-4mm}

\begin{center}
\setlength{\unitlength}{7mm}
\begin{picture}(10,6)
\put(5,1){\circle*{.3}}
\put(1,3){\circle*{.3}}
\put(3,3){\circle*{.3}}
\put(7,3){\circle*{.3}}
\put(9,3){\circle*{.3}}
\put(5,5){\circle*{.3}}
\put(5,1){\line(-2,1)4}
\put(5,1){\line(-1,1)2}
\put(5,1){\line(1,1)2}
\put(5,1){\line(2,1)4}
\put(5,5){\line(-2,-1)4}
\put(5,5){\line(-1,-1)2}
\put(5,5){\line(1,-1)2}
\put(5,5){\line(2,-1)4}
\put(4.85,.3){$0$}
\put(.3,2.8){$a_1$}
\put(2.3,2.8){$a_2$}
\put(4.65,2.8){$\cdots$}
\put(7.25,2.8){$a_{n-1}$}
\put(9.2,2.8){$a_n$}
\put(4.85,5.35){$1$}
\put(4.2,-.8){{\rm Fig.~1}}
\end{picture}
\end{center}

\vspace*{3mm}

\begin{theorem}\label{th3}
The lattice $\mathbf L(\mathbf V)$ is orthomodular if and only if $V$ does not contain an isotropic vector.
\end{theorem}

\begin{proof}
We use the fact that $\mathbf L(\mathbf V)$ is orthomodular if and only if $^\perp$ is a complementation. Since
\[
\dim(U+U^\perp)+\dim(U\cap U^\perp)=\dim U+\dim U^\perp=m
\]
for all $U\in L(\mathbf V)$, we have $U+U^\perp=V$ if and only if $U\cap U^\perp=\{\vec0\}$. Hence $^\perp$ is a complementation if and only if $U\cap U^\perp=\{\vec0\}$ for all $U\in L(\mathbf V)$. If $^\perp$ is not a complementation then there exists some $U\in L(\mathbf V)$ and some $\vec a\in V\setminus\{\vec 0\}$ with $\vec a\in U\cap U^\perp$. But then $\vec a$ is isotropic. If, conversely, $V$ possesses an isotropic element $\vec b$ then $\vec b\neq\vec0$ and $\vec b\in\langle\{\vec b\}\rangle\cap\langle\{\vec b\}\rangle^\perp$ and hence $\langle\{\vec b\}\rangle\cap\langle\{\vec b\}\rangle^\perp\neq\{\vec0\}$, i.e.\ $^\perp$ is not a complementation.
\end{proof}

If $V$ does not contain an isotropic vector then $\mathbf L(\mathbf V)$ is orthomodular because $^\perp$ is an orthocomplementation and $\mathbf L(\mathbf V)$ is modular. This is the case in the following example.

\begin{example}\label{ex1}
Assume $(q,m)=(3,2)$. Then the Hasse diagram of $\mathbf L(\mathbf V)$ looks as follows {\rm(}see Fig.~2{\rm)}:

\vspace*{-4mm}

\begin{center}
\setlength{\unitlength}{7mm}
\begin{picture}(8,6)
\put(4,1){\circle*{.3}}
\put(1,3){\circle*{.3}}
\put(3,3){\circle*{.3}}
\put(5,3){\circle*{.3}}
\put(7,3){\circle*{.3}}
\put(4,5){\circle*{.3}}
\put(4,1){\line(-3,2)3}
\put(4,1){\line(-1,2)1}
\put(4,1){\line(1,2)1}
\put(4,1){\line(3,2)3}
\put(4,5){\line(-3,-2)3}
\put(4,5){\line(-1,-2)1}
\put(4,5){\line(1,-2)1}
\put(4,5){\line(3,-2)3}
\put(3.55,.2){$\{\vec0\}$}
\put(.3,2.8){$A$}
\put(2.3,2.8){$B$}
\put(5.25,2.8){$C$}
\put(7.2,2.8){$D$}
\put(3.8,5.35){$V$}
\put(3.2,-.8){{\rm Fig.~2}}
\end{picture}
\end{center}

\vspace*{3mm}

where
\begin{align*}
A & :=\{(0,0),(0,1),(0,2)\}, \\
B & :=\{(0,0),(1,0),(2,0)\}, \\
C & :=\{(0,0),(1,1),(2,2)\}, \\
D & :=\{(0,0),(1,2),(2,1)\}.
\end{align*}
Hence $(L(\mathbf V),+,\cap)\cong\mathbf M_4$. Moreover,
\[
\begin{array}{c|cccc}
   U    & A & B & C & D \\
\hline
U^\perp & B & A & D & C\end{array}
\]
and hence $\mathbf L(\mathbf V)$ is orthomodular. This is in accordance with the fact that $V$ has no isotropic vector.
\end{example}

On the contrary, we have the following situation.

\begin{example}
Assume $(q,m)=(5,2)$. Then $\mathbf L(\mathbf V)$ is not orthomodular since $U^\perp=U$ for
\[
U=\{(0,0),(1,3),(2,1),(3,4),(4,2)\}.
\]
This is in accordance with the fact that $(1,2)$ is an isotropic vector of $V$.
\end{example}

It turns out that $\mathbf V$ contains an isotropic vector if $m$ is large enough. More precisely, if $m\geq p$ then there exists an isotropic vector in $\mathbf V$.

\begin{proposition}\label{prop1}
There exists a unique integer $m(q)$ with $1<m(q)\leq p$ such that $\mathbf L(\mathbf V)$ is orthomodular if and only if $m<m(q)$. In case $p=2$ we have $m(q)=2$.
\end{proposition}

\begin{proof}
By Theorem~\ref{th3}, we must show that there exists a unique integer $m(q)$ with $1<m(q)\leq p$ such that $V$ contains an isotropic vector $\vec a=(a_1,\ldots,a_m)$ if and only if $m\geq m(q)$. It is clear that in case $m=1$, $V$ cannot contain an isotropic vector. Obviously,
\[
m(q)=\min\{k\in\mathbb N\mid\text{there exist }a_1,\ldots,a_k\in({\rm GF}(q))\setminus\{0\}\text{ with }a_1^2+\cdots+a_k^2=0\}
\]
and since in case $m\geq p$ the vector $(1,\ldots,1,0,\ldots,0)$ with $1$ written down $p$ times is an isotropic vector of $V$, we have $m(q)\leq p$. In case $p=2\leq m$, $(1,1,0,\ldots,0)$ is an isotropic vector of $V$.
\end{proof}

\begin{example}
Assume $(q,m)=(2,2)$. Then the Hasse diagram of $\mathbf L(\mathbf V)$ looks as follows {\rm(}see Fig.~3{\rm)}:

\vspace*{-4mm}

\begin{center}
\setlength{\unitlength}{7mm}
\begin{picture}(6,6)
\put(3,1){\circle*{.3}}
\put(1,3){\circle*{.3}}
\put(3,3){\circle*{.3}}
\put(5,3){\circle*{.3}}
\put(3,5){\circle*{.3}}
\put(3,1){\line(-1,1)2}
\put(3,1){\line(0,1)4}
\put(3,1){\line(1,1)2}
\put(3,5){\line(-1,-1)2}
\put(3,5){\line(1,-1)2}
\put(2.55,.2){$\{\vec0\}$}
\put(.3,2.8){$A$}
\put(3.2,2.8){$B$}
\put(5.25,2.8){$C$}
\put(2.8,5.35){$V$}
\put(2.2,-.8){{\rm Fig.~3}}
\end{picture}
\end{center}

\vspace*{3mm}

where
\begin{align*}
A & :=\{(0,0),(0,1)\}, \\
B & :=\{(0,0),(1,0)\}, \\
C & :=\{(0,0),(1,1)\}.
\end{align*}
Hence $(L(\mathbf V),+,\cap)\cong\mathbf M_3$. Moreover,
\[
\begin{array}{c|cccc}
   U    & A & B & C \\
\hline
U^\perp & B & A & C
\end{array}
\]
and hence $\mathbf L(\mathbf V)$ is not orthomodular. This is in accordance with the fact that $(1,1)$ is an isotropic vector of $V$.
\end{example}

In some cases, we can find a smaller upper bound for $m(q)$.

\begin{theorem}\label{th5}
\
\begin{enumerate}[{\rm(i)}]
\item If $6|(p-1)(2p-1)$ then $m(q)\leq p-1$,
\item if $p>2$ and $24|(p+1)(p-1)$ then $m(q)\leq(p-1)/2$.
\end{enumerate}
\end{theorem}

\begin{proof}
In both cases we construct a suitable isotropic vector of $V$.
\begin{enumerate}[(i)]
\item If $6|(p-1)(2p-1)$ and $m\geq p-1$ then because of
\[
\sum_{i=1}^{p-1}i^2=\frac{(p-1)(2p-1)}6p=0
\]
the vector $(1,2,3,\ldots,p-1,0,\ldots,0)\in V$ is isotropic.
\item If $p>2$, $24|(p+1)(p-1)$ and $m\geq(p-1)/2$ then because of
\[
\sum_{i=1}^{(p-1)/2}i^2=\frac{(p+1)(p-1)}{24}p=0
\]
the vector $(1,2,3,\ldots,(p-1)/2,0,\ldots,0)\in V$ is isotropic.
\end{enumerate}
\end{proof}

\begin{corollary}
From {\rm(i)} of Theorem~\ref{th5} we obtain $m(q)\leq p-1$ if $p>2$ and $p\equiv-1\mod3$.
\end{corollary}

For small numbers $q$ we can compute $m(q)$ as follows.

\begin{example}
The following table shows the values of $m(q)$ for small $q$ and a corresponding isotropic vector of minimal dimension.
\[
\begin{array}{r|c|l}
 q & m(q) & \\
\hline
 2 &   2 & (1,1) \\
 3 &   3 & (1,1,1) \\
 4 &   2 & (1,1) \\
 5 &   2 & (1,2) \\
 7 &   3 & (1,2,3) \\
 8 &   2 & (1,1) \\
 9 &   2 & (1,x) \\
11 &   3 & (1,1,3) \\
13 &   2 & (2,3) \\
16 &   2 & (1,1) \\
17 &   2 & (1,4)
\end{array}
\]
Here we used ${\rm GF}(9)\cong\mathbb Z_3/(x^2+1)$ and neither ${\rm GF}(9)\cong\mathbb Z_3/(x^2+x-1)$ nor ${\rm GF}(9)\cong\mathbb Z_3/(x^2-x-1)$.
\end{example}

Our next task is the description of $\mathbf L(\mathbf V)$. We determine the number of $d$-dimensional linear subspaces of $\mathbf V$ as well as the numbers of covers in $\mathbf L(\mathbf V)$. For this reason, we introduce the numbers $a_n$ as follows:

Put $a_0:=1$ and
\[
a_n:=\prod_{i=1}^n(q^i-1)
\]
for all $n\in\mathbb N$.

Recall from \cite{Bi} that a {\em lattice} with $0$ is called {\em atomistic} if every of its elements is a join of atoms.

\begin{theorem}\label{th1}
Let $d\in\{0,\ldots,m\}$ and $\mathbf V$ be an $m$-dimensional vector space over ${\rm GF}(q)$. Then the following hold:
\begin{enumerate}[{\rm(i)}]
\item $\mathbf L(\mathbf V)$ is an atomistic modular lattice,
\item For every element $U\in L(\mathbf V)$, all maximal chains between $\{\vec0\}$ and $U$ have the same length {\rm(}{\em Chain condition}{\rm)},
\item $\mathbf V$ has exactly $a_m/(a_da_{m-d})$ $d$-dimensional linear subspaces.
\item If $d<m$ then every $d$-dimensional linear subspace of $\mathbf V$ is contained in exactly
\[
\frac{q^{m-d}-1}{q-1}=1+q+q^2+\cdots+q^{m-d-1}
\]
$(d+1)$-dimensional linear subspaces of $\mathbf V$.
\item If $d>0$ then every $d$-dimensional linear subspace of $\mathbf V$ contains exactly
\[
\frac{q^d-1}{q-1}=1+q+q^2+\cdots+q^{d-1}
\]
$(d-1)$-dimensional linear subspaces of $\mathbf V$.
\item The lattice $\mathbf L(\mathbf V)$ has exactly
\[
\frac{q^m-1}{q-1}=1+q+q^2+\cdots+q^{m-1}
\]
atoms, namely the one-dimensional linear subspaces of $\mathbf V$.
\item If $m=2$ then $(L(\mathbf V),+,\cap)\cong\mathbf M_{q+1}$.
\end{enumerate}
\end{theorem}

\begin{proof}
(i) and (ii) are well-known, (iii) is Theorem~1 in \cite{CL19}, (iv) is the special case $e:=d+1$ of Theorem~23 in \cite{CL19}, (v) is a special case of Theorem~1 in \cite{CL19} when considering the number of $(d-1)$-dimensional linear subspaces of a $d$-dimensional vector space over ${\rm GF}(q)$, (vi) is the special case $d:=0$ of (iv) and (vii) follows from (vi).
\end{proof}

The following concept will be used in the sequel.

\begin{definition}
An $m$-element subset $\{\vec b_1,\ldots,\vec b_m\}$ of an $m$-dimensional vector space $V$ over ${\rm GF}(q)$ is called an {\em orthogonal basis} of $\mathbf V$ if
\[
\vec b_i\vec b_j\left\{
\begin{array}{ll}
\neq0 & \text{if }i=j \\
=0    & \text{otherwise}
\end{array}
\right.
\]
\end{definition}

\begin{example}
\
\begin{itemize}
\item $\{(1,0,\ldots,0),(0,1,0,\ldots,0),\ldots,(0,\ldots,0,1)\}$ is an orthogonal basis of $\mathbf V$ for arbitrary $p$,
\item $\{(0,1,\ldots,1),(1,0,1,\ldots,1),\ldots,(1,\ldots,1,0)\}$ is an orthogonal basis of $\mathbf V$ if and only if $p|m-2$.
\end{itemize}
\end{example}

\begin{lemma}\label{lem1}
Let $B=\{\vec b_1,\ldots,\vec b_m\}$ be an orthogonal basis of $\mathbf V$ and $I\subseteq\{1,\ldots,m\}$. Then
\begin{enumerate}[{\rm(i)}]
\item $B$ is a basis of $\mathbf V$,
\item $\langle\{\vec b_i\mid i\in I\}\rangle^\perp=\langle\{\vec b_i\mid i\in\{1,\ldots,m\}\setminus I\}\rangle$.
\end{enumerate}
\end{lemma}

\begin{proof}
Assume $a_1,\ldots,a_m\in{\rm GF}(q)$ and put $\vec a:=a_1\vec b_1+\cdots+a_m\vec b_m$.
\begin{enumerate}[(i)]
\item If $\vec a=\vec0$ then $a_i\vec b_i\vec b_i=\vec a\vec b_i=\vec0\vec b_i=0$ for all $i=1,\ldots,m$ and hence $a_1=\cdots=a_m=0$ showing the independence of $\vec b_1,\ldots,\vec b_m$.
\item The following are equivalent:
\begin{align*}
             \vec a & \in\langle\{\vec b_i\mid i\in I\}\rangle^\perp, \\
     \vec a\vec b_i & =0\text{ for all }i\in I, \\
a_i\vec b_i\vec b_i & =0\text{ for all }i\in I, \\
                a_i & =0\text{ for all }i\in I, \\
             \vec a & \in\langle\{\vec b_i\mid i\in\{1,\ldots,m\}\setminus I\}\rangle.
\end{align*}
\end{enumerate}
\end{proof}

Denote by $\mathbf 2^k$ the finite Boolean lattice (Boolean algebra) having just $k$ atoms. In what follows we will check when $\mathbf L(\mathbf V)$ for an $m$-dimensional vector space $\mathbf V$ over ${\rm GF}(q)$ contains a subalgebra isomorphic to $\mathbf 2^k$ for some $k\leq m$.

\begin{theorem}\label{th4}
Let $\{\vec b_1,\ldots,\vec b_m\}$ be an orthogonal basis of $\mathbf V$. Then the subalgebra of $\mathbf L(\mathbf V)$ generated by $\{\langle\{\vec b_1\}\rangle,\ldots,\langle\{\vec b_m\}\rangle\}$ is isomorphic to $\mathbf2^m$. Since $\mathbf2^m$ contains subalgebras isomorphic to $\mathbf2^i$ for every $i=1,\ldots,m$, this is also true for $\mathbf L(\mathbf V)$.
\end{theorem}

\begin{proof}
Let $S$ denote the subuniverse of $\mathbf L(\mathbf V)$ generated by $\{\langle\{\vec b_1\}\rangle,\ldots,\langle\{\vec b_m\}\rangle\}$. Using Lemma~\ref{lem1} it is easy to see that $S=\{\langle\{\vec b_i\mid i\in I\}\rangle\mid I\subseteq\{1,\ldots,m\}\}$, that $S$ is a subuniverse of $\mathbf L(\mathbf V)$ and that $I\mapsto\langle\{\vec b_i\mid i\in I\}\rangle$ is an isomorphism from
\[
(2^{\{1,\ldots,m\}},\cup,\cap,I\mapsto\{1,\ldots,m\}\setminus I,\emptyset,\{1,\ldots,m\}\}
\]
to $(S,+,\cap,{}^\perp,\{\vec0\},V)$. This shows $(S,+,\cap,{}^\perp,\{\vec0\},V)\cong\mathbf2^m$. The rest of the proof is clear.
\end{proof}

As shown by Proposition~\ref{prop1}, if the dimension of $\mathbf V$ is small enough then $\mathbf L(\mathbf V)$ is an orthomodular lattice. Now we show when $\mathbf L(\mathbf V)$ contains an orthomodular lattice isomorphic to a horizontal sum of Boolean algebras also for an arbitrary dimension of $\mathbf V$ that is not a multiple of $p$.

Let $\mathbf L_i=(L_i,\vee_i,\wedge_i,{}'^{_i},0,1)$ ($i=1,2$) be non-trivial orthomodular lattices satisfying $L_1\cap L_2=\{0,1\}$. Then their {\em horizontal sum} $\mathbf L_1+\mathbf L_2=(L,\vee,\wedge,{}',0,1)$ is defined by
\begin{align*}
L & :=L_1\cup L_2, \\
x\vee y   & :=\left\{
\begin{array}{ll}
x\vee_iy & \text{if }i\in\{1,2\}\text{ and }x,y\in L_i \\
1        & \text{otherwise}
\end{array}
\right. \\
x\wedge y & :=\left\{
\begin{array}{ll}
x\wedge_iy & \text{if }i\in\{1,2\}\text{ and }x,y\in L_i \\
0          & \text{otherwise}
\end{array}
\right. \\
x'        & :=x'^{_i}\text{ if }i\in\{1,2\}\text{ and }x\in L_i
\end{align*}
($x,y\in L$). It is well-known that the horizontal sum of two orthomodular lattices is again an orthomodular lattice.

\begin{theorem}\label{th2}
Assume $p\nmid m$. Then there exists a subset $S$ of $V$ such that $(S,\subseteq,{}^\perp,\{\vec0\},V)$ is an orthomodular lattice isomorphic to the horizontal sum of the Boolean algebras $\mathbf2^m$ and $\mathbf2^2$. The presented set $S$ is a subuniverse of $\mathbf L(\mathbf V)$ if and only if $m=2$.
\end{theorem}

\begin{proof}
Put
\begin{align*}
       N & :=\{1,\ldots,m\}, \\
\vec e_i & :=(0,\ldots,0,1,0,\ldots,0),\text{ with }1\text{ at place }i,\text{ for all }i\in N \\
     U_I & :=\langle\{\vec e_i\mid i\in I\}\rangle\text{ for all }I\subseteq N, \\
       W & :=\langle\{(1,\ldots,1)\}\rangle, \\
       S & :=\{U_I\mid I\subseteq N\}\cup\{W,W^\perp\}.
\end{align*}
From Theorem~\ref{th4} we have that $\{U_I\mid I\subseteq N\}$ is a subuniverse of $\mathbf L(\mathbf V)$ and $I\mapsto U_I$ is an isomorphism from $(2^N,\cup,\cap,I\mapsto N\setminus I,\emptyset,N)$ to $(\{U_I\mid I\subseteq N\},+,\cap,{}^\perp,\{\vec0\},V)$. Clearly, $U_I\not\subseteq W,W^\perp\not\subseteq U_I$ for all $I\in2^N\setminus\{\emptyset,N\}$. This shows that in $(S,\subseteq,{}^\perp,\{\vec0\},V)$ the following hold for all $I\in2^N\setminus\{\emptyset,N\}$:
\begin{align*}
        W\vee U_I & =V, \\
      W\wedge U_I & =\{\vec0\}, \\
  W^\perp\vee U_I & =V, \\
W^\perp\wedge U_I & =\{\vec0\}.
\end{align*}
Moreover, $\dim W=1$ and $\dim W^\perp=m-1$. Since $p\nmid m$ we have $W\not\subseteq W^\perp$. In case $m=2$ we have for $i=1,2$
\begin{align*}
        W\vee U_{\{i\}} & =W+U_{\{i\}}=V, \\
      W\wedge U_{\{i\}} & =W\cap U_{\{i\}}=\{\vec0\}, \\
  W^\perp\vee U_{\{i\}} & =W+U_{\{i\}}=V, \\
W^\perp\wedge U_{\{i\}} & =W\cap U_{\{i\}}=\{\vec0\}
\end{align*}
and hence $S$ is a subuniverse of $\mathbf L(\mathbf V)$. If, however, $m>2$ then
\[
W+U_{\{1\}}\subsetneqq V=W\vee U_{\{1\}}
\]
since $\dim(W+U_{\{1\}})=2<m$. This shows that in this case $S$ is not a subuniverse of $\mathbf L(\mathbf V)$.
\end{proof}

The following  example shows a lattice of the form $\mathbf L(\mathbf V)$ that is not orthomodular, but contains a not Boolean orthomodular lattice as a subposet.

\begin{example}
Assume $(q,m)=(2,3)$. Then the Hasse diagram of $\mathbf L(\mathbf V)$ looks as follows {\rm(}see Fig.~4{\rm)}:

\vspace*{-4mm}

\begin{center}
\setlength{\unitlength}{7mm}
\begin{picture}(14,8)
\put(7,1){\circle*{.3}}
\put(1,3){\circle*{.3}}
\put(3,3){\circle*{.3}}
\put(5,3){\circle*{.3}}
\put(7,3){\circle*{.3}}
\put(9,3){\circle*{.3}}
\put(11,3){\circle*{.3}}
\put(13,3){\circle*{.3}}
\put(1,5){\circle*{.3}}
\put(3,5){\circle*{.3}}
\put(5,5){\circle*{.3}}
\put(7,5){\circle*{.3}}
\put(9,5){\circle*{.3}}
\put(11,5){\circle*{.3}}
\put(13,5){\circle*{.3}}
\put(7,7){\circle*{.3}}
\put(7,1){\line(-3,1)6}
\put(7,1){\line(-2,1)4}
\put(7,1){\line(-1,1)2}
\put(7,1){\line(0,1)6}
\put(7,1){\line(1,1)2}
\put(7,1){\line(2,1)4}
\put(7,1){\line(3,1)6}
\put(7,7){\line(-3,-1)6}
\put(7,7){\line(-2,-1)4}
\put(7,7){\line(-1,-1)2}
\put(7,7){\line(1,-1)2}
\put(7,7){\line(2,-1)4}
\put(7,7){\line(3,-1)6}
\put(1,3){\line(0,1)2}
\put(1,3){\line(1,1)2}
\put(1,3){\line(2,1)4}
\put(3,3){\line(-1,1)2}
\put(3,3){\line(2,1)4}
\put(3,3){\line(3,1)6}
\put(5,3){\line(-2,1)4}
\put(5,3){\line(3,1)6}
\put(5,3){\line(4,1)8}
\put(7,3){\line(-2,1)4}
\put(7,3){\line(2,1)4}
\put(9,3){\line(-3,1)6}
\put(9,3){\line(0,1)2}
\put(9,3){\line(2,1)4}
\put(11,3){\line(-3,1)6}
\put(11,3){\line(-2,1)4}
\put(11,3){\line(1,1)2}
\put(13,3){\line(-4,1)8}
\put(13,3){\line(-2,1)4}
\put(13,3){\line(-1,1)2}
\put(6.55,.2){$\{\vec0\}$}
\put(.3,2.8){$A$}
\put(2.3,2.8){$B$}
\put(4.3,2.8){$C$}
\put(7.3,2.8){$D$}
\put(9.3,2.8){$E$}
\put(11.3,2.8){$F$}
\put(13.3,2.8){$G$}
\put(.3,4.8){$H$}
\put(2.3,4.8){$I$}
\put(4.3,4.8){$J$}
\put(7.3,4.8){$K$}
\put(9.3,4.8){$L$}
\put(11.3,4.8){$M$}
\put(13.3,4.8){$N$}
\put(6.8,7.35){$V$}
\put(6.25,-.7){{\rm Fig.~4}}
\end{picture}
\end{center}

\vspace*{3mm}

where
\begin{align*}
A & :=\{(0,0,0),(0,0,1)\}, \\
B & :=\{(0,0,0),(0,1,0)\}, \\
C & :=\{(0,0,0),(0,1,1)\}, \\
D & :=\{(0,0,0),(1,0,0)\}, \\
E & :=\{(0,0,0),(1,0,1)\}, \\
F & :=\{(0,0,0),(1,1,0)\}, \\
G & :=\{(0,0,0),(1,1,1)\}, \\
H & :=\{(0,0,0),(0,0,1),(0,1,0),(0,1,1)\}, \\
I & :=\{(0,0,0),(0,0,1),(1,0,0),(1,0,1)\}, \\
J & :=\{(0,0,0),(0,0,1),(1,1,0),(1,1,1)\}, \\
K & :=\{(0,0,0),(0,1,0),(1,0,0),(1,1,0)\}, \\
L & :=\{(0,0,0),(0,1,0),(1,0,1),(1,1,1)\}, \\
M & :=\{(0,0,0),(0,1,1),(1,0,0),(1,1,1)\}, \\
N & :=\{(0,0,0),(0,1,1),(1,0,1),(1,1,0)\}.
\end{align*}
Moreover,
\[
\begin{array}{c|ccccccc}
   U    & A & B & C & D & E & F & G \\
\hline
U^\perp & K & I & M & H & L & J & N
\end{array}
\]
Since $C+C^\perp=M\neq V$, $^\perp$ is not a complementation and hence $\mathbf L(\mathbf V)$ is not orthomodular. This is in accordance with the fact that $(1,1,0)$ is an isotropic vector of $V$. We have $p\nmid m$. Hence we can apply Theorem~\ref{th2}. The set $S$ of Theorem~\ref{th2} equals $\{\{\vec0\},A,B,D,G,H,I,K,N,V\}$ and the Hasse diagram of the orthomodular lattice $(S,\subseteq,^\perp,\{\vec0\},V)$ is visualized in Fig.~5:

\vspace*{-4mm}

\begin{center}
\setlength{\unitlength}{7mm}
\begin{picture}(8,8)
\put(4,1){\circle*{.3}}
\put(1,3){\circle*{.3}}
\put(3,3){\circle*{.3}}
\put(5,3){\circle*{.3}}
\put(7,3){\circle*{.3}}
\put(1,5){\circle*{.3}}
\put(3,5){\circle*{.3}}
\put(5,5){\circle*{.3}}
\put(7,5){\circle*{.3}}
\put(4,7){\circle*{.3}}
\put(4,1){\line(-3,2)3}
\put(4,1){\line(-1,2)1}
\put(4,1){\line(1,2)1}
\put(4,1){\line(3,2)3}
\put(4,1){\line(3,4)3}
\put(4,7){\line(-3,-2)3}
\put(4,7){\line(-1,-2)1}
\put(4,7){\line(1,-2)1}
\put(4,7){\line(3,-2)3}
\put(4,7){\line(3,-4)3}
\put(1,3){\line(0,1)2}
\put(1,3){\line(1,1)2}
\put(3,3){\line(-1,1)2}
\put(3,3){\line(1,1)2}
\put(5,3){\line(-1,1)2}
\put(5,3){\line(0,1)2}
\put(3.55,.2){$\{\vec0\}$}
\put(.3,2.8){$A$}
\put(2.3,2.8){$B$}
\put(4.3,2.8){$D$}
\put(7.3,2.8){$G$}
\put(.3,4.8){$H$}
\put(2.3,4.8){$I$}
\put(4.3,4.8){$K$}
\put(7.3,4.8){$N$}
\put(3.8,7.35){$V$}
\put(3.25,-.7){{\rm Fig.~5}}
\end{picture}
\end{center}

\vspace*{3mm}

Since $D+G=M\notin S$, $S$ is not a subuniverse of $\mathbf L(\mathbf V)$. One can easily see that this lattice is the horizontal sum of the Boolean lattices $\mathbf2^3$ and $\mathbf2^2$.
\end{example}

If $V$ does not contain an isotropic vector then $\mathbf L(\mathbf V)$ is orthomodular because $^\perp$ is an orthocomplementation and $\mathbf L(\mathbf V)$ is modular. This is the case in the following example.

\begin{example}
In Example~\ref{ex1} we have $p\nmid m$. Hence we can apply Theorem~\ref{th2}. The set $S$ of Theorem~\ref{th2} coincides with $L(\mathbf V)$ and is therefore trivially a subuniverse of $\mathbf L(\mathbf V)$.
\end{example}

In the rest of paper, we will investigate the lattice $\mathbf L(\mathbf V)$ for two-dimensional vector spaces $\mathbf V$ over ${\rm GF}(q)$.

For $n>1$ let $\mathbf{MO}_n$ denote the modular ortholattice with the following Hasse diagram (see Fig.~6):

\vspace*{-4mm}

\begin{center}
\setlength{\unitlength}{7mm}
\begin{picture}(10,6)
\put(5,1){\circle*{.3}}
\put(1,3){\circle*{.3}}
\put(3,3){\circle*{.3}}
\put(7,3){\circle*{.3}}
\put(9,3){\circle*{.3}}
\put(5,5){\circle*{.3}}
\put(5,1){\line(-2,1)4}
\put(5,1){\line(-1,1)2}
\put(5,1){\line(1,1)2}
\put(5,1){\line(2,1)4}
\put(5,5){\line(-2,-1)4}
\put(5,5){\line(-1,-1)2}
\put(5,5){\line(1,-1)2}
\put(5,5){\line(2,-1)4}
\put(4.85,.3){$0$}
\put(.3,2.8){$a_1$}
\put(2.2,2.8){$a_1^\perp$}
\put(4.65,2.8){$\cdots$}
\put(7.25,2.8){$a_n$}
\put(9.2,2.8){$a_n^\perp$}
\put(4.85,5.35){$1$}
\put(4.2,-.8){{\rm Fig.~6}}
\end{picture}
\end{center}

\vspace*{3mm}

\begin{theorem}
Let $\mathbf V$ be a $2$-dimensional vector space over the field ${\rm GF}(q)$ for some $q=p^n$, without loss of generality assume $V=({\rm GF}(q))^2$, and put $M:=\{(x,y)\in\mathbb N^2\mid x\leq y\leq p/2\}$. Then
\begin{enumerate}[{\rm(i)}]
\item If there exists some $(x,y)\in M$ with $p|(x^2+y^2)$ then $\mathbf L(\mathbf V)$ is not orthomodular,
\item if $q=p$ then $\mathbf L(\mathbf V)$ is orthomodular if and only if $p\nmid(x^2+y^2)$ for all $(x,y)\in M$.
\item if $\mathbf L(\mathbf V)$ is orthomodular then $\mathbf L(\mathbf V)\cong\mathbf{MO}_{(q+1)/2}$.
\end{enumerate}
\end{theorem}

\begin{proof}
\
\begin{enumerate}[(i)]
\item If there exists some $(x,y)\in M$ with $p|(x^2+y^2)$ then $(x,y)$ is an isotropic vector of $V$ and hence $\mathbf L(\mathbf V)$ is not orthomodular according to Theorem~\ref{th3}.
\item Assume $q=p$. Then ${\rm GF}(q)\cong\mathbb Z_p$. According to Theorem~\ref{th3}, $\mathbf L(\mathbf V)$ is orthomodular if and only if $V$ does not contain an isotropic vector. Now $(a,b)\in\mathbb Z_p^2$ is isotropic if and only if $(a,b)\neq(0,0)$ and $a^2+b^2=0$ in $\mathbb Z_p$. If $(a,b)$ is isotropic then $a\neq0$ and $b\neq0$ in $\mathbb Z_p$. Since modulo $p$ all non-zero elements of $\mathbb Z_p$ are given by $1$ if $p=2$ and by $\pm1,\pm2,\pm3,\ldots,\pm(p-1)/2$ otherwise, all squares of non-zero elements are given modulo $p$ by $1$ if $p=2$ and by $1^2,2^2,3^2,\ldots,((p-1)/2)^2$ otherwise.
\item This follows from (vii) of Theorem~\ref{th1}.
\end{enumerate}
\end{proof}

For small $q$ we list all $2$-dimensional vector spaces $\mathbf V$ over ${\rm GF}(q)$ and indicate for which of them $\mathbf L(\mathbf V)$ is orthomodular.

\begin{example}
For $m=2$ we have
\[
\begin{array}{r|l|l|l}
 q & (L(\mathbf V),+,\cap) & \mathbf L(\mathbf V)    & \text{isotropic vector} \\
\hline
 2 & \cong\mathbf M_3      & \text{not orthomodular} & (1,1) \\
 3 & \cong\mathbf M_4      & \cong\mathbf{MO}_2      & \\
 4 & \cong\mathbf M_5      & \text{not orthomodular} & (1,1) \\
 5 & \cong\mathbf M_6      & \text{not orthomodular} & (1,2) \\
 7 & \cong\mathbf M_8      & \cong\mathbf{MO}_4      & \\
 8 & \cong\mathbf M_9      &\text{not orthomodular}  & (1,1) \\
 9 & \cong\mathbf M_{10}   &\text{not orthomodular}  & (1,x) \\
11 & \cong\mathbf M_{12}   & \cong\mathbf{MO}_6      & \\
13 & \cong\mathbf M_{14}   &\text{not orthomodular}  & (2,3) \\
16 & \cong\mathbf M_{17}   &\text{not orthomodular}  & (1,1) \\
17 & \cong\mathbf M_{18}   &\text{not orthomodular}  & (1,4)
\end{array}
\]
Here we used ${\rm GF}(9)\cong\mathbb Z_3/(x^2+1)$ and neither ${\rm GF}(9)\cong\mathbb Z_3/(x^2+x-1)$ nor ${\rm GF}(9)\cong\mathbb Z_3/(x^2-x-1)$.
\end{example}

Authors' addresses:

Ivan Chajda \\
Palack\'y University Olomouc \\
Faculty of Science \\
Department of Algebra and Geometry \\
17.\ listopadu 12 \\
771 46 Olomouc \\
Czech Republic \\
ivan.chajda@upol.cz

Helmut L\"anger \\
TU Wien \\
Faculty of Mathematics and Geoinformation \\
Institute of Discrete Mathematics and Geometry \\
Wiedner Hauptstra\ss e 8-10 \\
1040 Vienna \\
Austria, and \\
Palack\'y University Olomouc \\
Faculty of Science \\
Department of Algebra and Geometry \\
17.\ listopadu 12 \\
771 46 Olomouc \\
Czech Republic \\
helmut.laenger@tuwien.ac.at
\end{document}